\documentclass[11pt]{article}
\usepackage{amsmath,amsthm,amssymb}
\usepackage[hyperfootnotes=false]{hyperref}
\usepackage{color}
\usepackage{cancel}
\usepackage{titlesec}
\usepackage{enumerate}
\usepackage[normalem]{ulem}

\usepackage{geometry}
\usepackage{array}

\titleformat{\paragraph}
{\normalfont\normalsize\bfseries}{\theparagraph}{1em}{}
\titlespacing*{\paragraph}
{0pt}{3.25ex plus 1ex minus .2ex}{1.5ex plus .2ex}

\def\titlerunning#1{\gdef\titrun{#1}}
\makeatletter
\def\author#1{\gdef\autrun{\def\and{\unskip, }#1}\gdef\@author{#1}}
\def\address#1{{\def\and{\\\hspace*{18pt}}\renewcommand{\thefootnote}{}%
\footnote {#1}}%
\markboth{\autrun}{\titrun}}
\makeatother
\def\email#1{\hspace*{4pt}{\em e-mail}: #1}
\def\MSC#1{{\renewcommand{\thefootnote}{}%
\footnote{\emph{Mathematics Subject Classification (2020):} #1}}}
\def\keywords#1{\par\medskip
\noindent\textbf{Keywords:} #1}

\newtheorem{theorem}{Theorem}[section]

\newtheorem{prop}[theorem]{Proposition}
\newtheorem{cor}[theorem]{Corollary}
\newtheorem{lemma}[theorem]{Lemma}

\theoremstyle{definition}

\newtheorem{prob}[theorem]{Problem}

\newtheorem{remark}[theorem]{Remark}
\newtheorem{exa}[theorem]{Example}

\numberwithin{equation}{section}

\def\cA{\mathcal A}
\def\cB{\mathcal B}
\def\cC{\mathcal C}

\def\cJ{\mathcal J}
\def\cL{\mathcal L}

\def\cN{\mathcal N}
\def\cO{\mathcal O}
\def\cP{\mathcal P}
\def\cQ{\mathcal Q}

\def\cS{\mathcal S}
\def\cT{\mathcal T}

\def\cV{\mathcal V}

\def\cX{\mathcal X}
\def\cY{\mathcal Y}

\def\PG{{\rm PG}}
\def\AG{{\rm AG}}

\def\F{{\mathbb F}}

\def\PGL{{\rm PGL}}
\def\PSL{{\rm PSL}}

\def\rk{{\rm rk}}

\def\v{\boldsymbol v}
\def\w{\omega}
\def\a{\boldsymbol a}

\begin{document}


\baselineskip=16pt

\titlerunning{}

\title{On $4$-general sets in finite projective spaces}

\author{Francesco Pavese}

\date{}

\maketitle

\address{F. Pavese: Dipartimento di Meccanica, Matematica e Management, Politecnico di Bari, Via Orabona 4, 70125 Bari, Italy; \email{francesco.pavese@poliba.it}}


\MSC{Primary: 51E22 Secondary: 51E20; 94B05.}

\begin{abstract}
A {\em $4$-general set} in $\PG(n, q)$ is a set of points of $\PG(n, q)$ spanning the whole $\PG(n, q)$ and such that no four of them are on a plane. Such a pointset is said to be {\em complete} if it is not contained in a larger $4$-general set of $\PG(n, q)$. In this paper upper and lower bounds for the size of the largest and the smallest complete $4$-general set in $\PG(n, q)$, respectively, are investigated. Complete $4$-general sets in $\PG(n, q)$, $q \in \{3, 4\}$, whose size is close to the theoretical upper bound are provided. Further results are also presented, including a description of the complete $4$-general sets in projective spaces of small dimension over small fields and the construction of a transitive $4$-general set of size $3(q+1)$ in $\PG(5, q)$, $q \equiv 1 \pmod{3}$. 

\keywords{points in general position, cap.}
\end{abstract}

\section{Introduction}

Let $q$ be a prime power and let $\PG(n, q)$ or $\AG(n, q)$ denote the $n$-dimensional projective or affine space over the finite field $\F_q$. A {\em cap} is a set of points in $\PG(n, q)$ such that at most two of them are on a line, whereas a set of points in $\PG(n, q)$ such that at most $n$ in a hyperplane is known as an {\em arc}. These objects have been extensively studied due to their connections to coding theory see for instance \cite{HS}. More generally, following \cite{H1983}, a set $\cX$ in $\PG(n, q)$ is called {\em $(|\cX|; r, s, n, q)$-set} if the properties below are satisfied:
\begin{itemize}
\item[{\em i)}] each $s$-dimensional projective subspace contains at most $r$ points of $\cX$;
\item[{\em ii)}] $\cX$ spans the whole $\PG(n, q)$;
\item[{\em iii)}] there is an $(s+1)$-dimensional projective subspace containing $r+2$ points of $\cX$.
\end{itemize}
The term {\em $(r+2)$-general set}, $1 \le r \le n-1$, is also used to denote a $(|\cX|; r+1, r, n, q)$-set, see \cite{Bennett, TW}. Indeed, $\cX$ is an $(r+2)$-general set if any $r+2$ distinct points of $\cX$ are in general position. Hence a cap is a $3$-general set and an arc is an $(n+1)$-general set. An $(r+2)$-general set is called {\em complete} if it is not contained in a larger $(r+2)$-general set. The maximum size and the size of the smallest complete $(r+2)$-general set of $\PG(n, q)$ are respectively denoted by $M_{r+1}(n, q)$ and $T_{r+1}(n, q)$. An $(r+2)$-general set can be similarly defined in $\AG(n, q)$. In particular, an $(r+2)$-general set of $\AG(n, q)$ is also an $(r+2)$-general set of $\PG(n, q)$ and $M_{r+1}(n, q)$ provides an upper bound for the maximum size of an $(r+2)$-general set of $\AG(n, q)$. Here we focus on $4$-general sets in $\PG(n, q)$, i.e., sets of points of $\PG(n, q)$ no four on a plane, spanning the whole $\PG(n, q)$. Since property {\em iii)} given above is not essential in this case we will omit it. Observe that a $4$-general set with more than three points is also a cap of $\PG(n, q)$. 

In \cite{Bennett}, extending the methods used in \cite{EG}, an upper bound for the size of a $4$-general set in $\AG(n, q)$ was obtained. Recently in \cite{TW}, by using an arithmetic formulation, the authors acquired a significant improvement on this upper bound. Finally better upper bounds for the size of a $4$-general set in $\AG(n, q)$, $q \in \{2, 3\}$, are given in \cite{CP, HTW}. In particular, a $4$-general set in $\AG(n, q)$ has size at most 
\begin{align*}
& 2^{\frac{n+1}{2}} - 2, && \mbox{ if } q = 2 \mbox{ and } n \mbox{ odd, } \\ 
& \left\lfloor 2^{\frac{n+1}{2}} + \frac{1}{2} \right\rfloor - 2, && \mbox{ if } q = 2 \mbox{ and } n \mbox{ even, } \\
& \left\lceil 3^{\frac{n}{2}} \right\rceil, && \mbox{ if } q = 3, \\
& \frac{2q^{\frac{n}{2}}}{\sqrt{q-2}}, && \mbox{ if } q > 3.       
\end{align*}

The study of $4$-general sets is not only of geometrical interest, indeed it also has relevance in coding theory. A {\em $q$-ary linear code} $\cC$ of dimension $k$ and length $N$ is a $k$-dimensional vector subspace of $\F_q^N$, whose elements are called {\em codewords}. A {\em generator matrix} of $\cC$ is a matrix whose rows form a basis of $\cC$ as an $\F_q$-vector space. The {\em minimum distance} of $\cC$ is $d = \min\{d(u, 0) \mid u \in \cC, u \ne 0\}$, where $d(u, v)$, $u, v \in \F_q^N$, is the Hamming distance on $\F_q^N$. A vector $u$ is $\rho$-covered by $v$ if $d(u, v) \le \rho$. The {\em covering radius} of a code $\cC$ is the smallest integer $\rho$ such that every vector of $\F_q^N$ is $\rho$-covered by at least one codeword of $\cC$. A linear code with minimum distance $d$ and covering radius $\rho$ is said to be an $[N, k, d]_q$ $\rho$-code. For a code $\cC$, its dual code is $\cC^\perp = \{v \in \F_q^N \mid v \cdot c = 0, \, \forall c \in \cC\}$ (here ``$\cdot$'' is the Euclidean inner product). The dimension of the dual code $\cC^\perp$, or the codimension of $\cC$, is $N-k$. Any matrix which is a generator matrix of $\cC^\perp$ is called a {\em parity check matrix} of $\cC$. If $\cC$ is linear with parity check matrix $M$, its covering radius is the smallest $\rho$ such that every $w \in \F_q^{N-k}$ can be written as a linear combination of at most $\rho$ columns of $M$. Note that the distance of a code $\cC$ is the minimum number $d$ such that every $d-1$ columns of a parity-check matrix $H$ are linearly independent while there exist $d$ columns of $H$ that are linearly dependent. Therefore, by identifying the representatives of the points of a complete $4$-general set of $\PG(n, q)$ with columns of a parity check matrix of a $q$-ary linear code it follows that, apart from two sporadic exceptions, that are the $[7,1,7]_2$ repetition code and the $[23, 12, 7]_2$ Golay code (corresponding to a Veronese variety or a frame in $\PG(5, 2)$ and to the $23$-cap of $\PG(10, 2)$ admitting the Mathieu group $M_{23}$, respectively), complete $4$-general sets in $\PG(n, q)$ of size $k$ (with $k > n+1$) and non-extendable linear $[k, k-n-1, d]_q$ $3$-codes, $d \in \{5, 6\}$ are equivalent objects. Complete $4$-general sets are also examples of {\em 2-saturating sets}, that is, a set of points $\cS$ of a finite projective space such that each point is incident with a plane spanned by three points of $\cS$.

In this paper we investigate bounds and constructions of complete $4$-general sets in $\PG(n, q)$. In Section~\ref{sec1}, by using counting arguments, trivial upper and lower bounds for the size of the largest and the smallest complete $4$-general set, respectively, are given. If $q > 2$, the upper bound provided here is better than the one found in \cite[Theorem 5.1]{Bennett} in the affine case. In Section~\ref{sec2}, complete $4$-general sets in $\PG(n, q)$, $q \in \{3, 4\}$, whose size is close to the theoretical upper bound are exhibited. More precisely it is shown that 
\begin{align*}
& \frac{3^d+1}{2} \le M_3(2d-1, 3) \le \sqrt{\frac{3^{2d}-1}{2}}, \\
& \frac{2^{2d+1}+1}{3} \le M_3(2d, 4) \le \frac{\sqrt{8 \cdot 4^{2d+1} - 7} + 1}{6}, \\
& \frac{4^d-1}{3} \le M_3(2d-1, 4) \le \frac{\sqrt{8 \cdot 4^{2d} - 7} + 1}{6}.
\end{align*}
In Section~\ref{sec3}, further results are discussed. They include a description of the complete $4$-general sets in projective spaces of small dimension over small fields and the construction of a transitive $4$-general set of size $3(q+1)$ in $\PG(5, q)$, $q \equiv 1 \pmod{3}$. Here we speak of a $4$-general set $\cX$ of order $c q^r$ if the number of points of $\cX$ is a polynomial in $q$ with $c q^r$ as leading term, where $c$ is a positive constant.

\section{Preliminaries}

In this section we collect some auxiliary results that will be needed to our purposes. The first result is about the intersection of three linearly independent Hermitian curves of $\PG(2, q^2)$. Recall that a Hermitian curve of $\PG(2,q^2)$ is the set of zeros of a possibly degenerate Hermitian form. There are three possibilities, up to projective equivalence: a non-degenerate Hermitian curve, a Hermitian cone, and a repeated line, to which we refer as a rank three, a rank two and a rank one curve, respectively. From \cite{K, G}, the base locus of a pencil of Hermitian curves with no rank one curve has one of the following sizes:
\begin{itemize}
\item $q^2-q+1$, 
\item $q^2+1$, 
\item $q^2+q+1$, 
\item $q^2+2q+1$.
\end{itemize}
Moreover, such a pencil has at least $q-2$ non-degenerate Hermitian curves and all its members are non-degenerate Hermitian curves in the $q^2-q+1$ case only. 

\begin{lemma}\label{net}
Let $\cN$ be a net of Hermitian curves of $\PG(2, q^2)$, $q > 2$. If $\cN$ has at least one rank two curve and no rank one curve, then the base locus of $\cN$ is not empty.
\end{lemma}
\begin{proof}
Let $H$ and $H'$ be two curves of $\cN$, where $H$ has rank two, let $\cP$ be the pencil containing $H, H'$ and let $Z = H \cap H'$. Let $H'' \in \cN \setminus \cP$. Since the pencil containing $H$ and $H''$ has at least $q-2$ non-degenerate curves, we may fix a non-degenerate curve $H_1$ in this pencil. Assume by contradiction that the base locus of $\cN$ is empty, then $|Z \cap H_1| = 0$. Moreover $|H \cap H_1| \ge q^2+1$, whereas $|\tilde{H} \cap H_1| \ge q^2-q+1$, if $\tilde{H} \in \cP \setminus \{H\}$. Hence $q^3+1 = |H_1| = \sum_{\tilde{H} \in \cP} |\tilde{H} \cap H_1| \ge q(q^2-q+1)+q^2+1 = q^3+q+1$. A contradiction.  
\end{proof}

\begin{lemma}\label{cubic1}
In $\AG(2, q)$, $q$ even, $q \equiv 1 \pmod 3$, the cubic curve given by 
\begin{align*}
X^2Y+XY^2+X^2+Y^2+XY = 0
\end{align*}
has $q-3$ points. 
\end{lemma}
\begin{proof}
Some calculations show that the projective closure of the relevant cubic is absolutely irreducible with a node. Then, by \cite[Table 11.7]{H1}, it has $q$ points in $\PG(2, q)$, three of which lie on the line at infinity.  
\end{proof}

\begin{lemma}\label{cubic2}
In $\AG(2, q)$, $q$ even, $q \equiv 1 \pmod 3$, the number $R$ of points of the cubic curve given by 
\begin{align*}
\gamma (X^2Y+XY^2) + X^2 + X + Y^2 + Y + (\gamma + 1) XY = 0, \quad \gamma \in \F_{q} \setminus \{0, 1\},
\end{align*}
satisfies $q-2\sqrt{q}-2 \le R \le q+2\sqrt{q} -2$. 
\end{lemma}
\begin{proof}
Let $\cC$ be the projective closure of the relevant cubic. It can be easily checked that $\cC$ is non-singular. Then the claim follows by \cite[Section 11.10]{H1} and the fact that three points of $\cC$ are on the line at infinity.  
\end{proof}

\section{Bounds on the size of a $4$-general set in $\PG(n, q)$}\label{sec1}

In this section, by using counting arguments, we provide a trivial upper bound for $M_3(n, q)$ (Proposition~\ref{upper}) and a trivial lower bound for $T_3(n, q)$ (Proposition~\ref{lower}). 

\begin{prop}\label{upper}
Let $\cX$ be a $4$-general set of $\PG(n, q)$, then 
\begin{align*}
|\cX| \le \frac{\sqrt{8q^{n+1} + q^2 -6q +1} + q - 3}{2(q-1)}.
\end{align*}
Moreover equality occurs if and only if either $(q, n) = (2, 3)$ and in this case $\cX$ is projectively equivalent to an elliptic quadric of $\PG(3, q)$ or $(q, n) = (3, 4)$ and $\cX$ is projectively equivalent to the $11$-cap of $\PG(4, 3)$.
\end{prop}
\begin{proof}
Let $\cX$ be a $4$-general set of $\PG(n, q)$. A point of $\PG(n, q) \setminus \cX$ lies on at most one line secant to $\cX$. Since there are $|\cX|(|\cX| - 1)/2$ secant lines, it follows that 
\begin{align*}
(q-1) \frac{|\cX| (|\cX| - 1)}{2} + |\cX| \le \frac{q^{n+1} - 1}{q - 1}.
\end{align*}   
Hence 
\begin{align*}
(q-1) |\cX|^2 - (q-3) |\cX| - 2 \frac{q^{n+1} - 1}{q - 1} \le 0
\end{align*}
and the first part of the statement follows.

Assume equality holds. By considering the points of $\cX$ as the columns of a parity check matrix $H$ of a code $\cC$, since every point of $\PG(n, q)$ lies on a line secant to $\cX$, we have that the covering radius of $\cC$ is $2$. Moreover, every $4$ columns of $H$ are linearly independent and there are $5$ columns of $H$ that are linearly dependent. Hence $\cC$ has minimum distance $5$ and is a $2$-error correcting perfect code. It follows that either $\cC$ is the $[5, 1, 5]_2$ repetition code or the $[11, 6, 5]_3$ Golay code, see \cite{vanL}. The proof is now complete. 
\end{proof}

Note that, if $q > 2$, Proposition~\ref{upper} gives a better estimation on the maximum size of a $4$-general set in $\AG(n, q)$ than \cite[Theorem 5.1]{TW}. Moreover, it is known that $M_3(3, q)$ equals $5$, if $q \in \{2,3\}$, or $q+1$, if $q \ge 4$, \cite[Sections 21.2 and 21.3]{H2}.

\begin{prop}\label{lower}
Let $\cX$ be a complete $4$-general set of $\PG(n, q)$, then 
\begin{align*}
|\cX| > \frac{\sqrt[3]{6q^{n+1}-q^2-q+(q^2-5q+1)\sqrt[3]{6q^{n+1}-q^2-q}}}{q-1} + \frac{q-2}{q-1}.
\end{align*}
\end{prop}
\begin{proof}
Let $\cX$ be a complete $4$-general set of $\PG(n, q)$ of size $x$. Then every point of $\PG(n, q)$ lies on at least one plane spanned by three points of $\cX$. Therefore
\begin{align*}
\binom{x}{3}(q-1)^2 + \binom{x}{2} (q-1) + x \ge \frac{q^{n+1}-1}{q-1},
\end{align*}
which gives
\begin{align*}
F(x) = x^3 -3 \frac{q-2}{q-1} x^2 + \frac{2q^2-7q+11}{(q-1)^2} x - 6 \frac{q^{n+1}-1}{(q-1)^3} \ge 0.
\end{align*}
The polynomial $F(x)$ has precisely one real root, i.e., $x_0 = y_0 + \frac{q-2}{q-1}$, where 
\begin{align*}
& y_0 = \frac{\sqrt[3]{u + v} + \sqrt[3]{u - v}}{\sqrt[3]{2} \sqrt{3} (q-1)}, \\
& u = 3 \sqrt{3} (6q^{n+1}-q^2-q), v = \sqrt{27(6q^{n+1}-q^2-q)^2-4(q^2-5q+1)^3}.  
\end{align*}
Note that $y_0$ is positive and $G(y_0) = 0$, where  
\begin{align*}
G(y) = y^3 - \frac{q^2-5q+1}{(q-1)^2} y - \frac{6q^{n+1}-q^2-q}{(q-1)^3}.
\end{align*}
It follows that $y_0^3 - \frac{6q^{n+1}-q^2-q}{(q-1)^3} > y_0^3 - \frac{q^2-5q+1}{(q-1)^2} y_0 - \frac{6q^{n+1}-q^2-q}{(q-1)^3} = 0$ and hence $y_0 > \frac{\sqrt[3]{6q^{n+1}-q^2-q}}{q-1}$. More precisely $y_0^3 =  \frac{6q^{n+1}-q^2-q}{(q-1)^3} + \frac{q^2-5q+1}{(q-1)^2} y_0 > \frac{6q^{n+1}-q^2-q}{(q-1)^3} + \frac{q^2-5q+1}{(q-1)^2} \frac{\sqrt[3]{6q^{n+1}-q^2-q}}{q-1}$, as required. 
\end{proof}

If $q \in \{2,3,4\}$ or $q = 5$ and $n \in \{3, 4\}$, then $|\cX| > \frac{\sqrt[3]{6q^{n+1}-1}}{q-1}$ gives a better lower bound on the size of a complete $4$-general set of $\PG(n, q)$.

\section{Large complete $4$-general sets in finite projective spaces over small fields}\label{sec2}

\subsection{The cyclic model for $\PG(n, q)$}\label{cyclic}

Let $\w$ be a primitive element of $\F_{q^{n+1}}$ and hence $\w^{\frac{q^{n+1}-1}{q-1}}$ is a primitive element of $\F_q$. Since the finite field $\F_{q^{n+1}}$ is an $(n + 1)$-dimensional vector space over $\F_q$, the points of $\PG(n, q)$ can be identified with the field elements $1, \w, \dots, \w^{\frac{q^{n+1}-1}{q-1}-1}$ and will be denoted by 
\begin{align*}
\langle 1 \rangle, \langle \w \rangle, \dots, \langle \w^{\frac{q^{n+1}-1}{q-1}-1} \rangle. 
\end{align*}
Indeed, recall that in this representation two elements $x, y  \in \F_{q^{n+1}} \setminus \{0\}$ represent the same point of $\PG(n, q)$ if and only if $\frac{x}{y}  \in \F_q$, i.e., $x^{q-1} = y^{q-1}$. Moreover, three points $\langle x \rangle, \langle y \rangle, \langle z \rangle$ are collinear if and only if $ax + by = z$ for some $a, b \in \F_q$ and four points $\langle x \rangle, \langle y \rangle, \langle z \rangle, \langle t \rangle$ are coplanar if and only if $ax+by+cz = t$, for some $a, b, c \in \F_q$. Assume that $N$ divides $\frac{q^{n+1}-1}{q-1}$, where $\frac{q^{n+1}-1}{q-1} = N M$, and let $\xi$ be the projectivity induced by $\w^i \longmapsto \w^{i N}$. Then $\langle \xi \rangle$ is a cyclic group of order $M$ and, if $N, M > 1$, it acts semiregularly on points of $\PG(n, q)$ with point orbits:
\begin{align*}
\theta_i = \left\{\langle \w^{i + s N} \rangle \mid s = 0, \dots, M - 1 \right\}, \; i = 0, \dots, N-1.
\end{align*}
Since the projectivity induced by $\w^i \longmapsto \w^{i M}$ generates a group that centralizes $\langle \xi \rangle$ and permutes the orbits $\theta_i$, $i = 0, \dots, N-1$, the set $\theta_i$ is projectively equivalent to 
\begin{align*}
\theta_0 = \left\{\langle \w^i \rangle \mid \w^{i (q-1) M} = 1\right\}.
\end{align*}
In the two subsequent paragraphs we show that there are suitable integers $M$ such that $\theta_0$ is a complete $4$-general set of $\PG(n, q)$ when $(n, q) \in \{(2d-1, 3), (2d, 4)\}$.  

\subsubsection{$4$-general sets in $\PG(2d-1, 3)$}

Consider $\F_{q^d}$ as an $\F_q$-vector space of dimension $d$. Then an element $\a \in \F_q^d$ can be uniquely written as $\overline{\a} = (a_1, \dots, a_d) \in \F_q^d$. In this way, a point $(1, \a_1, \dots, \a_r)$ of $\AG(r, q^d)$ can be identified with the point $(1, \overline{\a_1}, \dots, \overline{\a_r})$ of $\AG(rd, q)$. This identification is the key ingredient of a well-known representation of the projective space $\PG(r, q^d)$ in $\PG(rd, q)$ independently found in \cite{A}, \cite{BB} and known as the {\em Andr\'e/Bruck-Bose representation}. In \cite{HTW} the authors proved that the set 
\begin{align*}
\cY = \{(1, \overline{a}, \overline{a^2}) \mid a \in \F_{3^d}\}
\end{align*}
is a $4$-general set in $\AG(2d, 3)$ and hence a $4$-general set of $\PG(2d, 3)$ of size $3^{d}$. Therefore
\begin{align*}
3^d \le M_3(2d, 3) \le \sqrt{\frac{3^{2d+1}-1}{2}}.
\end{align*}
In particular, it can be checked that $\cY$ can be extended to the unique $11$-cap of $\PG(4, 3)$ by adding two further points, if $d = 2$, whereas it is complete if $d = 3$. Hence the following problem arises.

\begin{prob}
Determine whether or not the $4$-general set $\cY \subset \PG(2d, q)$ is complete and enlarge it, if possible.
\end{prob}

In the same paper the authors also exhibited $4$-general sets of size $5$, $13$ and $33$ in the affine space over $\F_3$ of dimension $3$, $5$ or $7$, respectively. Here we show the existence of a complete $4$-general set of $\PG(2d-1, 3)$ of size $\frac{3^d+1}{2}$. In order to do that, consider the cyclic model described in subsection~\ref{cyclic}, where $n = 2d-1$ and $q = 3$. Set $N = 3^d-1$, $M = \frac{3^d+1}{2}$, so that   
\begin{align*}
\theta_0 = \left\{\langle \w^i \rangle \mid \w^{i(3^d+1)} = 1\right\}.
\end{align*}

\begin{theorem}\label{main1}
The set $\theta_0$ is a complete $4$-general set of $\PG(2d-1, 3)$, $d \ge 2$, of size $\frac{3^{d}+1}{2}$.
\end{theorem}
\begin{proof}
We will show that no three points of $\theta_0$ are on a line, no four points of $\theta_0$ are on a plane and then its completeness. 

Consider three points of $\theta_0$. Since $\theta_0$ is an orbit, we may assume that these points are $\langle 1 \rangle$, $\langle x \rangle$, $\langle y \rangle$, where $x^2 \ne 1$, $y^2 \ne 1$ and $x^2 \ne y^2$. If these three points were on a line then there are $a, b \in \F_3 \setminus \{0\}$ such that $a + bx = y$. Then $1 = y^{3^d+1} = (a + bx)^{3^d+1} = -1+ ab (x + x^{3^d})$ and hence $x^2+\frac{x}{ab}+1 = 0$, since $x^{3^d} = 1/x$. It follows that $x = \frac{1}{ab}$, a contradiction.

Similarly, assume by way of contradiction that $\langle 1 \rangle$, $\langle x \rangle$, $\langle y \rangle$, $\langle z \rangle$ are four points of $\theta_0$ on a plane. Here $x^2 \ne 1$, $y^2 \ne 1$, $z^2 \ne 1$, $x^2 \ne y^2$, $x^2 \ne z^2$ and $y^2 \ne z^2$. Then there are $a, b, c \in \F_3 \setminus \{0\}$ such that $a + bx + cy = z$ and $1 = z^{3^d+1} = (a+bx+cy)^{3^d+1} = ab(x+x^{3^d})+ac(y+y^{3^d})+bc(x^{3^d}y+xy^{3^d})$. Therefore $0 = aby(x^2+1)+acx(y^2+1)+bc(x^2+y^2)-xy = (x+bcy)(y+ac)(1+abx)$, a contradiction.

As for the completeness of $\theta_0$, we prove that for every $t \in \F_{3^{2d}} \setminus \{0\}$, there exist $x, y, z \in \theta_0$ such that $x+y+z = t$. Let $\PG(3, 3^{2d})$ be the three-dimensional projective space equipped with homogeneous projective coordinates $(X_1, X_2, X_3, X_4)$. In $\PG(3, 3^{2d})$ consider the Hermitian surfaces 
\begin{align*}
H_1: X_1^{3^d+1}+X_2^{3^d+1}+X_3^{3^d+1} = 0, \\
H_2: X_2^{3^d+1}+X_3^{3^d+1}+X_4^{3^d+1} = 0, \\
H_3: X_1^{3^d+1}+X_3^{3^d+1}+X_4^{3^d+1} = 0. 
\end{align*}
Let $\cB = H_1 \cap H_2 \cap H_3$. Since $\cB = \left\{(1, x, y, z) \mid x^{3^d+1} = y^{3^d+1} = z^{3^d+1} = 1\right\}$, in order to show that for every $t \in \F_{3^{2d}} \setminus \{0\}$, there exist $x, y, z \in \theta_0$ such that $x+y+z = t$, it is enough to prove that $|\cB \cap \pi_t| \ne 0$, where $\pi_t$ is the plane of $\PG(3, 3^{2d})$ given by $X_2+X_3+X_4 = t X_1$. Let $\overline{H}_i = H_i \cap \pi_t$. Then  
\begin{align*}
& \overline{H}_1: X_1^{3^d+1}+X_2^{3^d+1}+X_3^{3^d+1} = 0, \\
& \overline{H}_2: X_2^{3^d+1}+X_3^{3^d+1}+(t X_1 - X_2 - X_3)^{3^d+1} = 0, \\
& \overline{H}_3: X_1^{3^d+1}+X_3^{3^d+1}+(t X_1 - X_2 - X_3)^{3^d+1} = 0. 
\end{align*}
Note that $\overline{H}_2 - \overline{H}_1$ is a rank two Hermitian curve. Moreover, the net of Hermitian curves of $\pi_t$ spanned by $\overline{H}_1, \overline{H}_2, \overline{H}_3$ has no rank one curve. Indeed, it is easy to check that if $\lambda, \mu \in \F_{3^d}$, then the Hermitian matrices associated to $\overline{H}_2 + \lambda \overline{H}_1$ and $\mu \overline{H}_1 + \lambda \overline{H}_2 + \overline{H}_3$, namely 
\begin{align*}
\begin{pmatrix}
t^{3^d+1} + \lambda & -t & -t \\
-t^{3^d} & \lambda-1 & 1 \\
-t^{3^d} & 1 & \lambda-1
\end{pmatrix} 
\mbox{ and }
\begin{pmatrix}
(\lambda+1)t^{3^d+1} + \mu + 1 & -(\lambda + 1)t & -(\lambda + 1)t \\
-(\lambda + 1)t^{3^d} & 1-\lambda+\mu & \lambda + 1 \\
-(\lambda + 1)t^{3^d} & \lambda + 1 & \mu - \lambda - 1
\end{pmatrix},
\end{align*}
respectively, have rank at least $2$. The result now follows by Lemma~\ref{net}. 
\end{proof}

By combining Theorem~\ref{main1} and Proposition~\ref{upper} we obtain the following result.

\begin{cor}
\begin{align*}
\frac{3^d+1}{2} \le M_3(2d-1, 3) \le \sqrt{\frac{3^{2d}-1}{2}}.
\end{align*}
\end{cor}

\subsubsection{$4$-general sets in $\PG(2d, 4)$}

Take the geometric setting of the cyclic model for $n = 2d$ and $q = 4$. Fix $N = 2^{2d+1}-1$. Hence $M = \frac{2^{2d+1}+1}{3}$ and  
\begin{align*}
\theta_0 = \left\{\langle \w^i \rangle \mid \w^{i(2^{2d+1}+1)} = 1\right\}.
\end{align*}
The arguments employed to prove the following result are similar to those adopted in Theorem~\ref{main1}.

\begin{theorem}\label{main2}
The set $\theta_0$ is a complete $4$-general set of $\PG(2d, 4)$, $d \ge 1$, of size $\frac{2^{2d+1}+1}{3}$.
\end{theorem}
\begin{proof}
Suppose by contradiction that three points of $\theta_0$ are on a line. We may assume, without loss of generality, that one of these points is $\langle 1 \rangle$, being $\theta_0$ an orbit. Hence there are $a, b \in \F_4 \setminus \{0\}$ such that $a + bx = y$, where $\langle x \rangle, \langle y \rangle \in \theta_0$ and $x^3 \ne 1$, $y^3 \ne 1$, $x^3 \ne y^3$. Thus $1 = y^{2^{2d+1}+1} = (a + bx)^{2^{2d+1}+1} = ab^2 x^{2^{2d+1}} + a^2b x$, since $e^{2^{2d+1}} = e^2$, if $e \in \F_4 \setminus \{0\}$. By using the fact that $x^{2^{2d+1}} = 1/x$, we get $a^2b x^2 + x + ab^2 = 0$. Hence $x = \frac{e}{a^2b}$, with $e \in \F_4 \setminus \{0, 1\}$, which implies $x^3 = 1$, a contradiction.

If $\langle 1 \rangle$, $\langle x \rangle$, $\langle y \rangle$, $\langle z \rangle$ are four coplanar points of $\theta_0$, then $a + bx + cy = z$, where $a, b, c \in \F_4 \setminus \{0\}$ and $x^3 \ne 1$, $y^3 \ne 1$, $z^3 \ne 1$, $x^3 \ne y^3$, $x^3 \ne z^3$, $y^3 \ne z^3$. In this case $1 = z^{2^{2d+1}+1} = (a+bx+cy)^{2^{2d+1}+1} = ab^2 x^{2^{2d+1}}+ac^2 y^{2^{2d+1}}+a^2b x + a^2c y + b^2c x^{2^{2d+1}}y+ b c^2 xy^{2^{2d+1}}+1$. It follows that $0 = a(c^2x+b^2y)+a^2xy(bx + cy)+bc(cx^2+by^2) = \frac{1}{abc}(bx+cy)(bx + a)(cy+a)$, a contradiction.

In order to see that $\theta_0$ is not contained in a larger $4$-general set, it is enough to prove that for every $t \in \F_{4^{2d+1}} \setminus \{0\}$, there are $x, y, z \in \theta_0$ such that $x+y+z = t$. Let $\PG(3, 4^{2d+1})$ be the three-dimensional projective space equipped with homogeneous projective coordinates $(X_1, X_2, X_3, X_4)$. In $\PG(3, 4^{2d+1})$ consider the Hermitian surfaces 
\begin{align*}
H_1: X_1^{2^{2d+1}+1}+X_2^{2^{2d+1}+1} = 0, \\
H_2: X_1^{2^{2d+1}+1}+X_3^{2^{2d+1}+1} = 0, \\
H_3: X_1^{2^{2d+1}+1}+X_4^{2^{2d+1}+1} = 0. 
\end{align*}
Let $\cB = H_1 \cap H_2 \cap H_3$. Since $\cB = \left\{(1, x, y, z) \mid x^{2^{2d+1}+1} = y^{2^{2d+1}+1} = z^{2^{2d+1}+1} = 1\right\}$, demanding that for every $t \in \F_{4^{2d+1}} \setminus \{0\}$, there are $x, y, z \in \theta_0$ such that $x+y+z = t$, is equivalent to ask that $|\cB \cap \pi_t| \ne 0$, where $\pi_t$ is the plane of $\PG(3, 4^{2d+1})$ given by $X_2+X_3+X_4 = t X_1$. Denote by $\overline{H}_i$ the set $H_i \cap \pi_t$. Then  
\begin{align*}
& \overline{H}_1: X_1^{2^{2d+1}+1}+X_2^{2^{2d+1}+1} = 0, \\
& \overline{H}_2: X_1^{2^{2d+1}+1}+X_3^{2^{2d+1}+1} = 0, \\
& \overline{H}_3: X_1^{2^{2d+1}+1}+(t X_1 + X_2 + X_3)^{2^{2d+1}+1} = 0. 
\end{align*}
Straightforward calculations show that if $\lambda, \mu \in \F_{2^{2d+1}}$, then the Hermitian matrices associated to $\overline{H}_2 + \lambda \overline{H}_1$ and $\mu \overline{H}_1 + \lambda \overline{H}_2 + \overline{H}_3$, namely 
\begin{align*}
\begin{pmatrix}
\lambda + 1 & 0 & 0 \\
0 & \lambda & 0 \\
0 & 0 & 1
\end{pmatrix} 
\mbox{ and }
\begin{pmatrix}
t^{2^{2d+1}+1} + \lambda + \mu + 1 & t & t \\
t^{2^{2d+1}} & \mu + 1 & 1 \\
t^{2^{2d+1}} & 1 & \lambda + 1
\end{pmatrix},
\end{align*}
respectively, have rank at least $2$. Therefore Lemma~\ref{net} applies, since the net of Hermitian curves of $\pi_t$ spanned by $\overline{H}_1, \overline{H}_2, \overline{H}_3$ contains a rank two curve and has no rank one curve. The proof is now complete.
\end{proof}

As a consequence, the following result arises.

\begin{cor}
\begin{align*}
\frac{2^{2d+1}+1}{3} \le M_3(2d, 4) \le \frac{\sqrt{8 \cdot 4^{2d+1} - 7} + 1}{6}.
\end{align*}
\end{cor}

\subsection{$4$-general sets in $\PG(2d-1, 4)$}

Here we describe a class of complete $4$-general set of $\PG(2d-1, 4)$ of size $\frac{4^d-1}{3}$. In $\F_{4^d}^{2d}$, $d \ge 2$, consider the $2d$-dimensional $\F_4$-subspace $V$ given by the vectors 
\begin{align*}
\v(a, b) = (a, b, a^4, b^4, \dots, a^{4^{d-1}}, b^{4^{d-1}}), \quad a,b \in \F_{4^{d}}.
\end{align*}
Then $\PG(V)$ is a $(2d-1)$-dimensional projective space over $\F_4$. For $(a, b) \ne (0, 0)$, denote by $P(a, b)$ the point of $\PG(V)$ defined by the vector $\v(a,b)$. Let $\Pi_1$, $\Pi_2$ be the $(d-1)$-dimensional projective subspace consisting of the points $P(0, b)$, $P(a, 0)$, $a, b \in \F_{4^d} \setminus \{0\}$, respectively. Let us define the following pointsets of $\PG(V)$ 
\begin{align*}
\cV_\alpha = \{P(x, \alpha x^{-2}) \mid x \in \F_{4^d} \setminus \{0\}\}, \quad \alpha \in \F_{4^d} \setminus \{0\}.
\end{align*}
Let $\Phi$ be the projectivity of $\PG(V)$ induced by $\v(a, b) \in V \mapsto \v(\w a, \w^{-2} b) \in V$, where $\w$ is a primitive element of $\F_{4^d}$. It is straightforward to check the following properties.
\begin{enumerate}
\item[{\em i)}] $|\cV_\alpha| = \frac{4^d-1}{3}$.
\item[{\em ii)}] $\langle \Phi \rangle$ is a group of order $\frac{4^d-1}{3}$ acting regularly on points of $\Pi_1$, $\Pi_2$ and $\cV_{\alpha}$, $\alpha \in \F_{4^d} \setminus \{0\}$.  
\item[{\em iii)}] If $\delta, \alpha \in \F_{4^d} \setminus \{0\}$, then the projectivity induced by $\v(a, b) \mapsto \v(a, \delta b)$ fixes $\Pi_1$ and $\Pi_2$ and maps $\cV_{\alpha}$ to $\cV_{\delta \alpha}$.
\item[{\em iv)}] The sets $\cV_{\alpha}$, $\alpha \in \F_{4^d} \setminus \{0\}$ partition the pointset of $\PG(V) \setminus (\Pi_1 \cup \Pi_2)$.
\end{enumerate}

\begin{theorem}\label{main3}
The set $\cV_1$ is a complete $4$-general set of $\PG(2d-1, 4)$ of size $\frac{4^d-1}{3}$.
\end{theorem}
\begin{proof}
We first prove that no three points of $\cV_1$ are on a line and that no four points of $\cV_1$ are on a plane. 

Consider three points of $\cV_1$. Since $\cV_1$ is an orbit, we may assume that these points are $P(1, 1)$, $P(x, x^{-2})$, $P(y, y^{-2})$, where $x, y \in \F_{4^d} \setminus \{0\}$, $x^3 \ne 1$, $y^3 \ne 1$ and $x^3 \ne y^3$. If these three points were on a line, then there are $\lambda, \mu \in \F_4 \setminus \{0\}$ such that 
\begin{align*}
& \lambda + \mu x = y, \\ 
& \lambda + \mu x^{-2} = y^{-2}.
\end{align*}
Since $\lambda^3 = \mu^3 = 1$, the second equation writes as $(\lambda^2 x y + \mu^2 y + x)^2 = 0$. By substituting $y$, we obtain $\lambda \mu^2 x^2 = \lambda^2 \mu x + 1$. Hence $y^3 = (\lambda + \mu x)^{3} = 1+ x^3 + \lambda^2 \mu x + \lambda \mu^2 x^2 = x^3$, a contradiction.

Similarly, assume by way of contradiction that $P(1, 1)$, $P(x, x^{-2})$, $P(y, y^{-2})$, $P(z, z^{-2})$, where $x, y, z \in \F_{4^d} \setminus \{0\}$, are four points of $\cV_1$ on a plane. Here $x^3 \ne 1$, $y^3 \ne 1$, $z^3 \ne 1$, $x^3 \ne y^3$, $x^3 \ne z^3$ and $y^3 \ne z^3$. Then there are $\lambda, \mu, \eta \in \F_4 \setminus \{0\}$ such that 
\begin{align*}
& \lambda + \mu x + \eta y = z, \\ 
& \lambda + \mu x^{-2} + \eta y^{-2} = z^{-2}.
\end{align*}
Since $\lambda^3 = \mu^3 = \eta^3 = 1$, the second equation gives $(\lambda^2 x y z + \mu^2 y z + \eta^2 xz + xy)^2 = 0$. Taking into account the first equation, it follows that 
\begin{align*}
0 & = \lambda^2 x y (\lambda + \mu x + \eta y) + \mu^2 y (\lambda + \mu x + \eta y) + \eta^2 x (\lambda + \mu x + \eta y) + x y \\
& = (\eta^2 x + \mu^2 y) (\lambda^2 \mu x + 1) (\eta y + \lambda). 
\end{align*}
A contradiction.

In order to show that $\cV_1$ is complete we will prove that every point $Q$ of $\PG(V) \setminus \cV_1$ lies on a plane spanned by three points of $\cV_1$. In particular we claim that there are $x, y, z \in \F_{4^d} \setminus \{0\}$ such that $P(x, x^{-2}) + P(y, y^{-2}) + P(z, z^{-2}) = Q$. If $Q \in \Pi_1$, then we may assume $Q = P(0, 1)$. By Lemma~\ref{cubic1} there are $x, y \in \F_{4^d}$, with $(x, y) \ne (0,0)$, such that 
\begin{align*}
(x+y+xy)(x+y)+xy = 0. 
\end{align*}
Note that $x, y$ and $x+y$ are not zero. Set $z = x+y$. Dividing the latter equation by $xy(x+y)$ we obtain
\begin{align*}
1 = \frac{1}{x}+\frac{1}{y}+\frac{1}{x+y} = \frac{1}{x}+\frac{1}{y}+\frac{1}{z}. 
\end{align*}
Hence the following equations are satisfied
\begin{align*}
& x+y+z = 0, \\
& x^{-2}+y^{-2}+z^{-2} = 1,
\end{align*}
i.e., $P(x, x^{-2}) + P(y, y^{-2}) + P(z, z^{-2}) = P(0, 1)$. If $Q \in \Pi_2$ or $Q \in \cV_{\alpha} \setminus \{0, 1\}$, then we can take $Q$ as the point $P(1, \alpha)$, $\alpha \in \F_{4^d} \setminus \{1\}$. Let $x, y \in \F_{4^d} \setminus \{0\}$, such that the following holds
\begin{align*}
(\sqrt{\alpha} xy + x+y)(1+x+y)+xy = 0.
\end{align*}
Note that the existence of $x$ and $y$ is guaranteed by \cite[Theorem 7.16]{H1}, if $\alpha = 0$, and by Lemma~\ref{cubic2}, if $\alpha \ne 0$. Moreover $1+x+y \ne 0$. Let $z = 1+x+y$. Dividing the latter equation by $xy(1+ x+y)$ we get
\begin{align*}
\sqrt{\alpha} = \frac{1}{x}+\frac{1}{y}+\frac{1}{1+x+y} = \frac{1}{x}+\frac{1}{y}+\frac{1}{z}. 
\end{align*}
Hence the following equations are satisfied
\begin{align*}
& x+y+z = 1, \\
& x^{-2}+y^{-2}+z^{-2} = \alpha.
\end{align*}
Therefore $P(x, x^{-2}) + P(y, y^{-2}) + P(z, z^{-2}) = P(1, \alpha)$. 
\end{proof}

The following holds.

\begin{cor}
\begin{align*}
\frac{4^d-1}{3} \le M_3(2d-1, 4) \le \frac{\sqrt{8 \cdot 4^{2d} - 7} + 1}{6}.
\end{align*}
\end{cor}

\section{Further results}\label{sec3}

The study concerning $4$-general sets in spaces over the binary field is particularly intriguing and is the subject of a vast literature. Indeed, $4$-general sets in $\AG(n, 2)$ correspond to so called {\em Sidon sets} in $\F_2^n$. We refer to \cite{CP} for an upper bound for the maximum size of a Sidon set and \cite{N, RRW} and references therein for constructions of large or complete Sidon sets. The largest $4$-general sets in $\AG(n, 2)$, $n \le 6$, have been classified in \cite{Crager, CP}. By using similar arguments, that is by extending a simplex of reference, and with the aid of Magma \cite{magma}, it is not difficult to obtain the sizes, the order of the relative automorphism groups and a description of the projectively distinct complete $4$-general sets in $\PG(n, 2)$, $n \le 6$, which we provide in Table~\ref{Tab1}. 

In a similar way it can be seen that in $\PG(3, q)$, $q \in \{3, 4, 5, 7, 8\}$ and $\PG(4, 3)$ the complete $4$-general sets are those described in Table~\ref{Tab2}. 

\begin{table}
    \centering
     \setlength{\leftmargini}{0.4cm}
    \begin{tabular}{m{1cm} | m{1cm} | m{1cm} | m{10cm}}
$n$ & size & \#Aut &  \\
\hline
\hline
$3$ & $5$ & $120$ & an elliptic quadric (or a frame) \\
\hline
 & & $720$ & a frame \\
$4$ & $6$ & & \\
& & $120$ & the union of an elliptic quadric in a solid $\Pi$ and a point outside $\Pi$ \\
\hline
& $7$ & $5040$ & a Veronese variety (or a frame) \\
$5$ & & & \\
& $8$ & $144$ & the union of two elliptic quadrics lying in two distinct solids, $\Pi$, $\Pi'$, and such that $\Pi \cap \Pi'$ is a line secant to both quadrics \\
\hline
$6$ & $11$ & 48 & Example~\ref{PG(6,2)} 
    \end{tabular}
    \caption{\label{Tab1}Complete $4$-general sets in $\PG(n, 2)$, $n \in \{3,4,5,6\}$.}
\end{table}

\begin{table}
    \centering
     \setlength{\leftmargini}{0.4cm}
    \begin{tabular}{m{1cm} | m{1cm} | m{1cm} | m{1cm} | m{10cm}}
$n$ & $q$ & size & \#Aut &  \\
\hline
\hline
$3$ & $3$ & $5$ & $120$ & a frame \\
\hline
$3$ & $4$ & $5$ & $120$ & a twisted cubic (or a frame) \\
\hline
$3$ & $5$ & $6$ & $120$ & a twisted cubic \\
\hline
$3$ & $7$ & $8$ & $336$ & a twisted cubic \\
\hline
 & & $7$ & $6$ & $\{(1,t,t^2,t^3) \mid t \in \F_q \setminus \{0, 1\}\} \cup \{(0,1,1,0)\}$ \\
$3$ & $8$ & & & \\
 & & $9$ & $504$ & a twisted cubic \\
\hline
$4$ & $3$ & $11$ & $7920$ & the $11$-cap
    \end{tabular}
    \caption{\label{Tab2}Complete $4$-general sets in $\PG(n, q)$.}
\end{table}

\begin{remark}\label{oss}
Let $\cS$ be a frame of $\PG(2d, 2)$. Some trivial calculations show that the $(2d+1)(d+1)$ lines that are secant to $\cS$ cover, besides $\cS$, a set $\cS'$ consisting of $(2d+1)(d+1)$ points that are contained in a hyperplane $\Lambda$. In particular $\cS'$ is formed by the points on the $d(2d+1)$ lines of $\Lambda$ that are secant to a frame of $\Lambda$.  
\end{remark}

\begin{exa}\label{PG(6,2)}
Let $\PG(6, 2)$ equipped with projective homogeneous coordinates $(X_1, \dots, X_7)$. Consider the frame $\cS$ given by 
\begin{align*}
& \{(1,0,0,0,0,0,0), (0,1,0,0,0,0,0), (0,0,1,0,0,0,0), (0,0,0,1,0,0,0), \\
& \;\; (0,0,0,0,1,0,0), (0,0,0,0,0,1,0), (0,0,0,0,0,0,1), (1,1,1,1,1,1,1)\}.
\end{align*}
Clearly, a plane of $\PG(6, 2)$ meets $\cS$ in at most $3$ points. The $28$ lines that are secant to $\cS$ cover, besides $\cS$, a set $\cS'$ consisting of $28$ points that are contained in the hyperplane 
\begin{align*}
\Lambda: \sum_{i = 1}^7 X_i = 0. 
\end{align*}
Taking into account Remark~\ref{oss}, $\cS'$ is formed by the points lying on the $21$ lines that are secant to the following frame of $\Lambda$:
\begin{align*}
& \{(0,1,1,1,1,1,1), (1,0,1,1,1,1,1), (1,1,0,1,1,1,1), \\
& \;\; (1,1,1,0,1,1,1), (1,1,1,1,1,0,1,1), (1,1,1,1,1,0,1), (1,1,1,1,1,1,0)\}.
\end{align*} 
The $35$ points of $\Lambda \setminus \cS'$ are exactly those of the Klein quadric 
\begin{align*}
\cQ^+(5, 2): \sum_{i, j = 1, \; i \le j}^{7} X_i X_j = 0.
\end{align*}
Therefore the $56$ planes of $\PG(6, 2)$ intersecting $\cS$ in precisely $3$ points meet $\Lambda$ in a line external to $\cQ^+(5, 2)$. Moreover, these planes cover the $64$ points of $\PG(6, 2) \setminus \Lambda$. It follows that $\cS$ can be extended to a larger $4$-general set by adding any point of $\cQ^+(5, 2)$. On the other hand, it is not difficult to check that at most $3$ points $P_1, P_2, P_3$ of $\cQ^+(5, 2)$ can be selected in such a way that $\cS \cup \{P_1, P_2, P_3\}$ is a $4$-general set of $\PG(6, 2)$. In particular, the plane spanned by $P_1, P_2, P_3$ has to be contained in $\cQ^+(5, 2)$.      
\end{exa}

\subsection{$4$-general sets in $\PG(4, q)$}

A set $\cX$ of $\PG(4, q)$ is called an {\em NMDS-set} of $\PG(4, q)$, if the following hold:
\begin{itemize}
\item[{\em i)}] every $4$ points of $\cX$ generate a solid;
\item[{\em ii)}] there exist $5$ points of $\cX$ lying on a solid;
\item[{\em iii)}] every $6$ points of $\cX$ generate $\PG(4, q)$.
\end{itemize}
Therefore $4$-general sets of $\PG(4, q)$ with no more than $5$ points on a solid, and NMDS-sets of $\PG(4, q)$ are equivalent object. The reader is referred to \cite{DL, Giulietti} for more details. The largest known NMDS-sets arise from elliptic curves; from their sizes we get 
\begin{align*}
M_3(4, q) \ge 
\begin{cases} 
q + \lceil 2 \sqrt{q} \rceil & \mbox{ if } q = p^r, r \ge 3 \mbox{ odd, } p | \lceil 2 \sqrt{q} \rceil, \\ 
q + \lceil 2 \sqrt{q} \rceil + 1 & \mbox{ otherwise}. 
\end{cases}
\end{align*}

\subsection{$4$-general sets in $\PG(5, q)$}

In $\PG(3, q)$ the notions of arc and $4$-general set coincide. Moreover, if $q \ge 4$, an arc has at most $q+1$ points and an arc of maximum size is projectively equivalent to the pointset
\begin{align*}
\{(1, t, t^{2h}, t^{2h+1}) \mid t \in \F_q\} \cup \{(0,0,0,1)\},
\end{align*}
where $\gcd(n, h) = 1$ \cite[Theorems 21.2.3, 21.3.15]{H2}. If $h = 1$ such a pointset is called {\em twisted cubic} and if $q$ is odd every $(q + 1)$-arc of $\PG(3, q)$ is a twisted cubic. We remark the existence of an arc of size $2q+1$ in $\PG(3, q^3)$ not contained in a twisted cubic, see \cite{KV}. Here we show that if $q \equiv 1 \pmod 3$, then it is possible to glue together three distinct twisted cubics in order to obtain a $4$-general set of size $3(q+1)$. Assume $q \equiv 1 \pmod 3$ and set $\xi \in \F_q$ such that $\xi^2+\xi+1 = 0$. Consider in $\F_{q^2}^6$ the $6$-dimensional $\F_q$-subspace $V$ given by the vectors 
\begin{align*}
(x, x^q, y, y^q, z, z^q), \quad x, y, z \in \F_{q^2}.
\end{align*}
Then $\PG(V)$ is a $5$-dimensional projective space over $\F_q$. Let $G$ be the subgroup of $\PGL(V)$ generated by the projectivities induced by 
\begin{align*}
\begin{pmatrix}
1 & 0 & 0 & 0 & 0 & 0 \\
0 & 1 & 0 & 0 & 0 & 0 \\
0 & 0 & \xi & 0 & 0 & 0 \\
0 & 0 & 0 & \xi & 0 & 0 \\
0 & 0 & 1 & 0 & \xi^2 & 0 \\
0 & 0 & 0 & 1 & 0 & \xi^2 
\end{pmatrix},
\begin{pmatrix}
1 & 0 & 0 & 0 & 0 & 0 \\
0 & t^3 & 0 & 0 & 0 & 0 \\
0 & 0 & t & 0 & 0 & 0 \\
0 & 0 & 0 & t^2 & 0 & 0 \\
0 & 0 & 0 & 0 & t & 0 \\
0 & 0 & 0 & 0 & 0 & t^2 
\end{pmatrix}, \; t \in \F_{q^2}, t^{q+1} = 1,
\end{align*}
where the matrices act on the left. Hence $G$ is a group of order $3(q+1)$. It fixes the line 
\begin{align*}
\ell = \{(1,x,0,0,0,0) \mid x \in \F_{q^2}, x^{q+1} = 1\}
\end{align*}
and induces a line spread on the three-space 
\begin{align*}
\{(0,0,y,y^q,z,z^q) \mid y, z \in \F_{q^2}, (y, z) \ne (0,0)\}.
\end{align*}
For more details on the action of $G$ see also \cite{C}. 

Let $P = (1,1,1,1,0,0)$ and let $\cX = P^G$. We claim that $\cX$ is a $4$-general set. Observe that $\cX$ is the union of the three twisted cubics:
\begin{align*}
& \cC_1 = \{(1, t^3, t, t^2, 0, 0) \mid t \in \F_{q^2}, t^{q+1} = 1\}, \\
& \cC_2 = \{(1, t^3, \xi t, \xi t^2, t, t^2) \mid t \in \F_{q^2}, t^{q+1} = 1\}, \\
& \cC_3 = \{(1, t^3, \xi^2 t, \xi^2 t^2, -t, -t^2) \mid t \in \F_{q^2}, t^{q+1} = 1\}.
\end{align*}
The twisted cubic $\cC_i$ is contained in a solid, say $\Pi_i$, $i = 1,2,3$. Since the solids $\Pi_1$, $\Pi_2$, $\Pi_3$ pairwise intersect in $\ell$ and the line $\ell$ is an imaginary chord of $\cC_i$, $i = 1,2,3$, it follows that a plane through $\ell$ has at most one point in common with $\cX$. Assume by contradiction that there is a plane $\sigma$ containing $4$ points of $\cX$, then necessarily $\sigma$ intersects exactly one of the three solids, say $\Pi_{i_0}$, in a line that is secant to the twisted cubic $\cC_{i_0}$. Since $G$ permutes in a single orbit the three solids and the $q+1$ points of $\cC_{i_0}$, we may assume that $i_0 = 1$ and that the line of $\sigma$ secant to $\cC_{i_0}$ pass through $P$. Then there are $t_i \in \F_{q^2}$, $t_i^{q+1} = 1$, $t_1 \ne 1$, such that
\begin{align}\label{rank}
\rk
\begin{pmatrix}
1 & 1 & 1 & 1 & 0 & 0 \\
1 & t_1^3 & t_1 & t_1^2 & 0 & 0 \\
1 & t_2^3 & \xi t_2 & \xi t_2^2 & t_2 & t_2^2 \\
1 & t_3^3 & \xi^2 t_3 & \xi^2 t_3^2 & -t_3 & -t_3^2 
\end{pmatrix} = 3,
\end{align} 
Some straightforward but tedious calculations show that \eqref{rank} cannot occur and hence a contradiction arises. We have proved the following. 
\begin{prop}
There is a transitive $4$-general set in $\PG(5, q)$, $q \equiv 1 \pmod 3$, of size $3(q+1)$ that is the union of three twisted cubic.
\end{prop} 

We mention below some interesting examples of $4$-general sets in $\PG(5, q)$, $q \in \{5, 16\}$, with projective homogeneous coordinates $(X_1, \dots, X_6)$. It would be nice to determine whether or not they belong to an infinite family.

\begin{exa}
In $\PG(5, 5)$ denote by $\cA$ the following pointset 
\begin{align*}
\{(1,x,y,2xy,x^2+2y^2, x^3+x^2y+xy^2-y^3) \mid x,y \in \F_5\}. 
\end{align*}
The set $\cA$ si left invariant by a subgroup of $\PGL(6, 5)$ of order $300$ fixing $(0,0,0,0,0,1)$ and acting transitively on the points of $\cA$. Some straightforward but tedious calculations show that if $(x_i, y_i) \in \F_5^2 \setminus \{(0,0)\}$, $|\{(x_1,y_1), (x_2, y_2), (x_3, y_3)\}| = 3$, then the rank of the matrix 
\begin{align*}
\begin{pmatrix}
1 & 0 & 0 & 0 & 0 & 0 \\
1 & x_1 & y_1 & 2x_1y_1 & x_1^2+2y_1^2 & x_1^3+x_1^2y_1+x_1y_1^2-y_1^3 \\
1 & x_2 & y_2 & 2x_2y_2 & x_2^2+2y_2^2 & x_2^3+x_2^2y_2+x_2y_2^2-y_2^3 \\
1 & x_3 & y_3 & 2x_3y_3 & x_3^2+2y_3^2 & x_3^3+x_3^2y_3+x_3y_3^2-y_3^3 
\end{pmatrix}
\end{align*}
equals $4$. Hence $\cA$ is a $4$-general set. Moreover it can be enlarged by adding any point of the three lines $X_1 = X_2 = X_3 = X_5 - \alpha X_4 = 0$, $\alpha \in \{0, 2, 3\}$. Therefore 
\begin{align*}
\cA \cup \{(0,0,0,1,0,0), (0,0,0,1,2,0), (0,0,0,0,0,1)\}
\end{align*}
is a complete $4$-general set of $\PG(5, 5)$. This yields $28 \le M_3(5,5) \le 44$.
\end{exa}

\begin{exa}
Let $\xi$ be a primitive element of $\F_{16}$ with minimal polynomial $X^4+X+1 = 0$. Let $\cB$ be the orbit of the point $(1,\xi^5,\xi,\xi^2,0,0) \in \PG(5, 16)$ under the action of the group $K$ induced by the matrices
\begin{align*}
\begin{pmatrix}
1 & 0 & 0 & 0 & 0 & 0 \\
ac & a^2 & 0 & 0 & 0 & 0 \\
a^2c^2 & 0 & a^4 & 0 & 0 & 0 \\
a^3c^3 & a^4c^2 & a^5c & a^6 & 0 & 0 \\
a^4c^4 & 0 & 0 & 0 & a^8 & 0 \\
a^5c^5 & a^6c^4 & 0 & 0 & a^9c & a^{10} \\
\end{pmatrix}, \quad a, c \in \F_{16}, c^5 = 1.
\end{align*} 
Here matrices act on the left. It turns out that $\cB$ is a $4$-general set of $\PG(5, 16)$ that can be completed by adding the points $(0,0,0,1,0,0), (0,0,0,0,0,1)$. Note that $K$ is the unique subgroup of order $80$ of the stabilizer of the point $(0,0,0,0,0,1)$ in the projectivity group isomorphic to $\PSL(2, q^2)$ fixing the normal rational curve $\{(1,t,t^2,t^3,t^4,t^5) \mid t \in \F_{16}\} \cup \{(0,0,0,0,0,1)\}$. In particular, $\cB \cup \{(0,0,0,0,0,1)\}$ is the union of $5$ normal rational curves of $\PG(5, 16)$ through the point $(0,0,0,0,0,1)$. It follows that $82 \le M_3(5, 16) \le 386$.
\end{exa}

\subsection{$4$-general sets in higher dimensions}

For large $q$, to the best of our knowledge, the existence of a $4$-general set in $\PG(n, q)$ of order $q^{\frac{n-1}{2}}$ is only known in the cases when $n = 3$, where there are $(q+1)$-arcs and $n = 7$. Indeed, in $\PG(7, q)$ there exists a complete $4$-general set of size $q^3+1$. Consider the $8$-dimensional $\F_q$-subspace $V$ given by the vectors 
\begin{align*}
(x, y, z, t), \quad x, t \in \F_{q}, y, z \in \F_{q^3}.
\end{align*}
Then $\PG(V)$ is a $7$-dimensional projective space over $\F_q$. Let $\cO$ be the subset of $\PG(V)$ given by 
\begin{align*}
\{(1,x,x^{q^2+q},x^{q^2+q+1}) \mid x \in \F_{q^3}\} \cup \{(0,0,0,1)\}.
\end{align*}
The set $\cO$ has been introduced in \cite{Kantor, Lunardon1}. A plane meets $\cO$ in at most three points and there is a subgroup $G \simeq \PGL(2, q^3)$ of $\PGL(V)$ leaving invariant $\cO$, see \cite{Lunardon0, Lunardon2}. By \cite{Feng}, the group $G$ has four orbits on points of $\PG(V)$ given by
\begin{align*}
\cO = (1,0,0,0)^G, \quad \cO_1 = (0,0,1,0)^G, \quad \cO_2 = (1,0,0,1)^G, \quad \cO_3 = (1,0,\alpha,\alpha)^G,
\end{align*}
for some $\alpha \in \F_q$ such that $X^2-X-\alpha = 0$ is irreducible over $\F_q$. It is easily seen that the orbit $\cO_2$ consists of points lying on secant lines to $\cO$. With the same notation used in \cite{Feng}, let $g_1$ and $g_2$ be the projectivities of $G$ induced by 
\begin{align*}
\begin{pmatrix}
1 & b_1 \\
-1 & 0
\end{pmatrix} \mbox{ and }
\begin{pmatrix}
1 & b_2 \\
-1 & 0
\end{pmatrix}, 
\end{align*}
respectively, where $b_1, b_2$ are fixed elements in $\F_{q^3} \setminus \{0\}$, such that $b_1+b_1^q+b_1^{q^2} = 0$ and $b_2+b_2^q+b_2^{q^2} = -2$. Then $(0,0,1,0)^{g_1} = (0,b_1,b_1,0)$ and $(1,0,\alpha,\alpha)^{g_2} = (-3\alpha, \alpha (b_2+1), \alpha (b_2+1), \alpha-2)$. Since the points $(0,b_1,b_1,0)$ and $(-3\alpha, \alpha (b_2+1), \alpha (b_2+1), \alpha-2)$ both belong to the plane spanned by $(1,0,0,0), (1,1,1,1), (0,0,0,1)$ we infer that the planes having three points in common with $\cO$ cover all the points of $\PG(V)$ and hence $\cO$ is a complete $4$-general set. 

In $\PG(6, q)$ it is possible to obtain $4$-general sets of order $q^2$ by considering the intersection of $\cO$ with a hyperplane. Indeed a hyperplane intersects $\cO$ in either $q^2+q+1$ or $q^2+1$ or $q^2-q+1$ points, see \cite{Feng}. A further construction of a $4$-general set of size $q^2+1$ in $\AG(6, q)$ can be realized by using the Andr{\'e}/Bruck-Bose representation of a $(q^2+1)$-arc of $\PG(3, q^2)$ with the plane at infinity disjoint from it. 

\begin{remark}
Consider the Grassmann graph $\cJ_q(n+1, 2)$ whose vertices are the lines of $\PG(n, q)$ and two vertices are adjacent if the corresponding lines have a non-trivial intersection. Hence $\cJ_q(n+1, 2)$ is a strongly regular graph. Let $\cX$ be a $4$-general set of $\PG(n, q)$ and let $\cL$ be the set of lines that are secant to $\cX$. Then the subgraph of $\cJ_q(n+1, 2)$ induced by $\cL$ is strongly regular and isomorphic to the triangular graph $\cT_{|\cX|}$.
\end{remark}

We conclude the paper with the following open problems. 

\begin{prob}\label{grande}
Does there exist a $4$-general set in $\PG(n, q)$ of order $q^{\frac{n-1}{2}}$? 
\end{prob}
Problem~\ref{grande} has a positive answer in the cases when $n = 3$ and $n = 7$.

\begin{prob}\label{piccolo}
Does there exist a complete $4$-general set in $\PG(n, q)$ of order $c q^{\frac{n-2}{3}}$, where $c$ is a positive constant? 
\end{prob}

\begin{prob}
Classify the complete $4$-general sets with at most $4$ points on a three-space or equivalently the linear codes with minimum distance $6$ and covering radius $3$. 
\end{prob}

\bigskip

\smallskip
{\footnotesize
\noindent\textit{Acknowledgments.}
This work was supported by the Italian National Group for Algebraic and Geometric Structures and their Applications (GNSAGA-- INdAM).
}

\end{document}